\documentclass[12]{article}

\usepackage{amsmath}
\usepackage{amsfonts}
\usepackage{latexsym}
\usepackage{epsfig}
\usepackage{graphicx}
\usepackage{url}
\usepackage{multicol}

\usepackage{amssymb}

\usepackage{geometry}

\newtheorem{theorem}{Theorem}

\newtheorem{corollary}[theorem]{Corollary}

\newtheorem{lemma}[theorem]{Lemma}

\newtheorem{proposition}[theorem]{Proposition}
\newtheorem{remark}[theorem]{Remark}

\newenvironment{proof}[1][Proof]{\textbf{#1.} }{\ \rule{0.5em}{0.5em}}

 \usepackage[colorlinks=true]{hyperref}
 \usepackage{eucal}
 \usepackage[spanish]{babel}
 \usepackage[pdftex]{color}

\author{%
Luis Boza\hspace{-1.4mm}\vspace{-.42mm} \\
\small Departamento de Matem\'atica Aplicada I, Universidad de Sevilla, Spain\hspace{-1.4mm}\vspace{-.42mm} \\
\small boza@us.es}

\title{
Exact values and bounds for Ramsey numbers of $C_4$ versus a star graph}

\date{}

\begin{document}
\maketitle
\begin{abstract}
We study Ramsey numbers of the form $R(C_4,K_{1,n})$. We determine the eight previously unknown values of $R(C_4,K_{1,n})$ for $n\le 38$. In particular, we show that $R(C_4,K_{1,27})=33$ and $R(C_4,K_{1,n})=n+7$ for $28\le n\le 33$ and for $n=37$. We also establish new general inequalities relating different values of this function. Specifically, if $m\equiv 2\pmod 6$ with $m\ge 8$, then $R(C_4,K_{1,m^2+3})\le m^2+m+4$, and for all positive integers $a$ and $b$, either $R(C_4,K_{1,a})\ge a+b$ or $R(C_4,K_{1,a+b})\le a+2b$. As consequences, we obtain the functional inequalities $f(2n-f(n)+1)\ge n$ and $f(f(n)+1)\le 2f(n)-n+2$, where $f(n)=R(C_4,K_{1,n})$.

\end{abstract}

\section{Preliminaries}

Let $G$ and $H$ be two graphs. The notation $G \nsupseteq H$ means that $H$ is not isomorphic to a subgraph of $G$. The complement of $G$ is denoted by $\overline{G}$. $\Delta(G)$ and $\delta(G)$ represent the maximum and minimum degrees of $G$, respectively. The number of edges in $G$ is denoted by $e(G)$, and $V(G)$ represents its vertex set. If $A \subseteq V(G)$, then $G[A]$ refers to the subgraph of $G$ induced by $A$. For any $v \in V(G)$, the degree of $v$ in $G$ is denoted by $d_G(v)$.

We will use the following notation from \cite{R}: $K_k$ is a complete graph on $k$ vertices, the graph $kG$ is formed by $k$ disjoint copies of $G$, $G\cup H$ stands for vertex disjoint union of graphs, and the join graph $G+H$ is obtained by adding all of the edges between vertices of $G$ and $H$ to $G\cup H$. $C_k$ is a cycle on $k$ vertices, $K_{1,k}=K_1+kK_1$ is a star on $k+1$ vertices, and $W_k=K_1+C_{k-1}$ is a wheel on $k$ vertices.

The Ramsey number $R(H_1,H_2)$ is the smallest integer $N$ such that for every graph $G$ with $N$ vertices, either $G \supseteq H_1$ or $\overline{G} \supseteq H_2$.

Several works in the literature have studied $R(C_4, K_{1,n})$ and $R(C_4, W_n)$. Both values are equivalent for $n \geq 6$ \cite{ZhaBC1}. Consequently, all results proven for $R(C_4, K_{1,n})$ also apply to $R(C_4, W_n)$.

For simplicity, let $f(n):=R(C_4,K_{1,n})$.

Note that for a graph $\overline{G}$ to avoid $K_{1,n}$ is equivalent to having $\Delta(\overline{G}) \le n-1$. Hence, $f(n)$ is the smallest integer $N$ such that there is no graph $G$ with $N$ vertices where $G \nsupseteq C_4$ and $\delta(G)=N-1- \Delta(\overline{G})\geq N - n$.

The following lemmas will be used to prove the main results:

\begin{lemma} \label{Chen} \cite{Chen}
$f(n-1) \geq f(n) - 2$.
\end{lemma}

\begin{lemma} \label{par3} \cite{Par3}
For $n, m \geq 2$, $f(n) \leq n + \lceil\sqrt{n}\rceil + 1$ and $f(m^2 + 1) \leq m^2 + m + 2$.
\end{lemma}

This last lemma can be expressed as:
\begin{corollary} \label{cop3}
For $n \geq 2$, $f(n) \leq n + \lceil\sqrt{n-1}\rceil + 1$.
\end{corollary}

In Section \ref{gr}, we prove the following results:
\begin{itemize}
\item If $m \equiv 2 \pmod{6}$ and $m \ge 8$, then $f(m^2+3)\le m^2+m+4$.
\item $f(a)\ge a+b$ or $f(a+b)\le a+2b$. As consequences, we obtain
$f(2n-f(n)+1)\ge n$ and $f(f(n)+1)\le 2f(n)-n+2$.
\end{itemize}
These results will be used in Section \ref{sv} to obtain exact values and bounds on $f(n)$ for small $n$.

\section{General Results} \label{gr}

The main results of this section are:

\begin{theorem} \label{teo2}
Let $m\equiv 2 \,(\text{mod}\,\, 6)$ with $m\ge 8$. Then $f(m^2+3)\le m^2+m+4$.
\end{theorem}

\begin{proof}
Suppose there exists a graph $G$ with $m^2+m+4$ vertices such that
$G\nsupseteq C_4$ and $G\nsupseteq K_{1,m^2+3}$. Then $\delta(G)\ge m+1$.

Fix a vertex $v\in V(G)$. Since $G\nsupseteq C_4$, every vertex of $N_G(v)$ is adjacent
to at most one other vertex of $N_G(v)$; otherwise two such neighbours together with $v$
would form a copy of $C_4$. Moreover, if a vertex
$w\in V(G)\setminus (N_G(v)\cup\{v\})$ were adjacent to two vertices of $N_G(v)$, then
$v$, $w$, and these two vertices would also form a $C_4$. Hence every vertex of
$N_G(v)$ is adjacent to at least $d_G(v)-2$ vertices outside $N_G(v)\cup\{v\}$.

If $d_G(v)\ge m+2$, then there are at least
$(m+2)(m-1)$ edges joining $N_G(v)$ to
$V(G)\setminus (N_G(v)\cup\{v\})$. Since
$|V(G)\setminus (N_G(v)\cup\{v\})|
\le m^2+m+4-1-(m+2)<(m+2)(m-1)$,
some vertex outside $N_G(v)\cup\{v\}$ must be adjacent to at least two vertices of
$N_G(v)$, a contradiction. Therefore $d_G(v)\le m+1$.

Since $\delta(G)\ge m+1$, it follows that every vertex of $G$ has degree exactly $m+1$.

For each $v\in V(G)$, let $t_v$ denote the number of edges with both endpoints in $N_G(v)$. $t_v$ is the number of triangles to which $v$ belongs, and the total number of triangles in $G$ is $(\sum_{v\in V(G)}t_v)/3$.

Clearly, there are $2t_v$ vertices in $N_G(v)$ adjacent to $m-1$ vertices in $N_{\overline{G}}(v)$ and $m+1-2t_v$ vertices in $N_G(v)$ adjacent to $m$ vertices in $N_{\overline{G}}(v)$.

Since $0\le t_v\le\lfloor (m+1)/2\rfloor=m/2$, and each vertex in $N_G(v)$ is adjacent to at most one other vertex in $N_G(v)$,
the number of edges with one endpoint in $N_G(v)$ and the other in $N_{\overline{G}}(v)$ is $2t_v(m-1)+(m+1-2t_v)m$. Since $|N_{\overline{G}}(v)|=m^2+2$, it follows that $2t_v(m-1)+(m+1-2t_v)m\le m^2+2$ and hence $t_v\ge m/2-1$.

Assuming that $t_v=m/2$ for every $v\in V(G)$,  the number of triangles in $G$ would be $\left(\sum_{v\in V(G)}t_v\right)\!/3=m(m^2+m+4)/6.$ Since $m\equiv 2 \,(\text{mod}\, 6)$, let $a=(m-2)/6$, which is an integer. The total number of triangles then becomes $36a^3+42a^2+20a+10/3$, which is not an integer, leading to a contradiction. Therefore, there must exist a vertex $v_0\in V(G)$ such that $t_{v_0}=m/2-1$.

Let $F=G[V(G)\setminus(N_G(v_0)\cup\{v_0\})]$, so $|V(F)|=m^2+2$. Let $u_1,\ldots,u_{m+1}$ be the vertices in $N_G(v_0)$, with $u_{2i-1}$ adjacent to $u_{2i}$ for $1\le i\le m/2-1$. Define $A_j=N_G(u_j)\cap V(F)$ for $1\le j\le m+1$. Then, $|A_j|=m-1$ for $1\le j\le m-2$ and $|A_j|=m$ for $m-1\le j\le m+1$. Since $G\nsupseteq C_4$, if $i\ne j$, then $A_i\cap A_j=\emptyset$. Thus $|\bigcup_{i=1}^{m+1}A_i|=\sum_{i=1}^{m+1}|A_i|=(m-2)(m-1)+3m=m^2+2$, implying $\bigcup_{i=1}^{m+1}A_i=V(F)$.

Since $|A_1|=m-1$ is odd and $\delta(F[A_1])\le 1$, there exists $w_1\in A_1$ that is not adjacent to any other vertex in $A_1$. If $w_1$ is adjacent to two vertices in $A_i$, for some $i$, those two vertices  together with $w_1$ and $u_i$ would form a $C_4$, leading to $w_1$ being adjacent to at most one vertex in $A_i$.

If $w_1$ were adjacent to a vertex in $A_2$, then, since $u_1$ and $u_2$ are adjacent, a $C_4$ would be formed. Thus, $w_1$ is adjacent to $u_1$ and at most one vertex in each $A_i$ with $3\le i\le m+1$. Consequently, $w_1$ is adjacent to at most $m$ vertices in $G$, which leads to a contradiction and yields the desired result.
\end{proof}

\begin{lemma} \label{l5}
For every $n\in \mathbb{N}$ and every $m$ such that $n\le m\le f(n)-1$, there exists a graph $G$ with $m$ vertices such that $G\nsupseteq C_4$ and $\delta(G)=m-n$.
\end{lemma}
\begin{proof}
Let $G$ be a graph with $m$ vertices and the minimum possible number of edges among all graphs satisfying $G\nsupseteq C_4$ and $\delta(G)\ge m-n$. Suppose that $\delta(G)\ge m-n+1$. Let $x$ be an edge of $G$. Then $G-x\nsupseteq C_4$, and
$\delta(G-x)\ge \delta(G)-1\ge m-n$. This contradicts the minimality of the number of edges of $G$. Therefore, $\delta(G)=m-n$, which proves the result.
\end{proof}

\begin{theorem} 
$f(a) \ge a+b$ or $f(a+b) \le a+2b$.
\end{theorem}

\begin{proof}
Assume that $f(a+b) > a+2b$. By Lemma \ref{l5}, there exists a graph $G$ with $a+2b$ vertices such that $G \nsupseteq C_4$ and $\overline{G} \nsupseteq K_{1, a+b}$ and $\delta(G)=b$. Hence $\delta(G) \ge b$. Let $v_0 \in V(G)$ with $d_G(v_0)=b$ and define $F = G\left[N_{\overline{G}}(v_0)\right]$.

The graph $F$ has $a+b-1$ vertices and satisfies $F \nsupseteq C_4$. If there existed a vertex $w \in V(F)$ adjacent in $G$ to two vertices in $N_G(v_0)$, then these three vertices together with $v_0$ would form a $C_4$, a contradiction. Therefore, for every $w \in V(F)$ we have $d_F(w) \ge d_G(w) - 1 \ge b-1$, and consequently $d_{\overline{F}}(w)=(a+b-1) - 1 - d_F(w) \le a-1$. 
Thus $\Delta(\overline F)\le a-1$, and therefore
$\overline F\nsupseteq K_{1,a}$.
Since $|V(F)|=a+b-1$ and $F\nsupseteq C_4$, it follows that
$f(a)\ge a+b$.
\end{proof}

\begin{corollary} \label{pro}
$f(2n + 1 - f(n)) \ge n$.
\end{corollary}
\begin{proof}
Let $a = 2n - f(n) + 1$ and $b = f(n) - n - 1$. Then $f(a+b) = a + 2b + 1$, and hence $f(a) \ge a+b = n$. This completes the proof.
\end{proof}

\begin{corollary} \label{pro2}
$f(f(n)+1) \le 2f(n)-n+2$.
\end{corollary}
\begin{proof}
Let $a =n$ and $b = f(n)-n+1$. Then $f(a) = a + b - 1$, and hence $f(f(n)+1)=f(a+b) \le a+2b = 2f(n)-n+2$. This completes the proof.
\end{proof}

\section{Small values} \label{sv}

In this section, we present the new exact values or bounds of $f(n)$ for small values of $n$, marked with an asterisk (*) in the following tables.
\vspace{-8mm}
\begin{center}
$$\left.\begin{array}{|c|c|c|c|c|c|c|c|c|c|c|c|c|c|c|c|c|c|c|c|c|c|c|c|c|c|c|}\hline
n & \!\mbox{\small 1}\! & \!\mbox{\small 2}\! & \!\mbox{\small 3}\! & \!\mbox{\small 4}\! & \!\mbox{\small 5}\! & \!\mbox{\small 6}\! & \!\mbox{\small 7}\! & \!\mbox{\small 8}\! & \!\mbox{\small 9}\! & \!\mbox{\small 10}\! & \!\mbox{\small 11}\! & \!\mbox{\small 12}\! & \!\mbox{\small 13}\! & \!\mbox{\small 14}\! & \!\mbox{\small 15}\! & \!\mbox{\small 16}\! & \!\mbox{\small 17}\! & \!\mbox{\small 18}\! & \!\mbox{\small 19}\!& \!{\mbox{\small 20}}\! & \!{\mbox{\small 21}}\! & \!{\mbox{\small 22}}\! & \!{\mbox{\small 23}}\! & \!{\mbox{\small 24}}\! \\ \hline
f(n) &\!4\! & \!4\! & \!6\! & \!7\! & \!8\! & \!9\! & \!11\! & \!12\! & \!13\! & \!14\! & \!16\! & \!17\! & \!18\! & \!19\! & \!20\! & \!21\! & \!22\! & \!23\! & \!24\!& \!25\! & \!27\! & \!28\! & \!29\! & \!30\! \\ \hline
Ref. & & \multicolumn{3}{|c|}{\small\cite{ChH2}} & \!{\small\cite{Clan}}\! & \!{\small\cite{Par3}}\! & \!{\small\cite{FRS4}}\! & \multicolumn{2}{|c|}{\small\cite{Tse1}} & \multicolumn{2}{|c|}{\small\cite{Par3}} & \!{\small\cite{Tse1}}\! & \multicolumn{2}{|c|}{\small\cite{DyDz1}} & \!{\small\cite{ZhaBC1}}\! & \multicolumn{2}{|c|}{\small\cite{Par3}} & \!{\small\cite{ZhaBC1}}\! & \multicolumn{2}{|c|}{\small\cite{WuSR}} & \!{\small\cite{Par5}}\! & \!{\small\cite{SunSh}}\! & \!{\small\cite{Par5}}\! & \!{\small\cite{WuSZR}}\! \\ \hline
\end{array}\right.$$
$$\left.\begin{array}{|c|c|c|c|c|c|c|c|c|c|c|c|c|c|c|c|c|c|c|c|c|c|c|c|c|c|c|c|}\hline
n  & \!{\mbox{\small 25}}\! & \!{\mbox{\small 26}}\!& \!{\mbox{\small 27}}\! & \!{\mbox{\small 28}}\! & \!{\mbox{\small 29}}\! & \!{\mbox{\small 30}}\! & \!{\mbox{\small 31}}\! & \!{\mbox{\small 32}}\! & \!{\mbox{\small 33}}\! & \!{\mbox{\small 34}}\!& \!{\mbox{\small 35}}\! & \!{\mbox{\small 36}}\! & \!{\mbox{\small 37}}\! & \!{\mbox{\small 38}}\! & \!{\mbox{\small 39}}\! & \!{\mbox{\small 40}}\! & \!{\mbox{\small 41}}\! & \!{\mbox{\small 42}}\! & \!{\mbox{\small 43}}\!   & \!{\mbox{\small 44}}\! & \!{\mbox{\small 45}}\! \\ \hline
f(n) & \!31\! & \!32\! & \!{\bf 33}\! & \!{\bf 35}\! & \!{\bf 36}\! & \!{\bf 37}\! & \!{\bf 38}\! & \!{\bf 39}\! & \!{\bf 40}\! & \!41 & \!42\! & \!43\! & \!{\bf 44}\! & \!45\! & /46\!  & \!47\! & \!49\! & /50 & \!51\!  & /52  & \!53\! \\ \hline
Ref. & \multicolumn{2}{|c|}{\small\cite{Par3}} & \multicolumn{7}{|c|}{\bf *} & \!{\small\cite{WuSR}}\! & \!{\small\cite{WuSR}}\! & \!{\small\cite{WuSR}}\! & {\bf *} & \!{\small\cite{ZhaCC2}}\! & \!{\small\cite{WuSR}}\! & \!{\small\cite{ZhaCC2}}\! & \multicolumn{2}{|c|}{\small\cite{Par3}} & \!{\small\cite{WuSR}}\!  & \!{\small\cite{Par3}}\! & \!{\small\cite{ZhaCC3}}\! \\ \hline
\end{array}\right.$$
$$\left.\begin{array}{|c|c|c|c|c|c|c|c|c|c|c|c|c|c|c|c|c|c|c|c|c|c|c|}\hline
n  & \!{\mbox{\small 46}}\!  & \!{\mbox{\small 47}}\! & \!{\mbox{\small 48}}\! & \!{\mbox{\small 49}}\! & \!{\mbox{\small 50}}\! & \!{\mbox{\small 51}}\! & \!{\mbox{\small 52}}\! & \!{\mbox{\small 53}}\! & \!{\mbox{\small 54}}\! & \!{\mbox{\small 55}}\! & \!{\mbox{\small 56}}\! & \!{\mbox{\small 57}}\! & \!{\mbox{\small 58}}\! & \!{\mbox{\small 59}}\! & \!{\mbox{\small 60}}\! & \!{\mbox{\small 61}}\! & \!{\mbox{\small 62}}\! & \!{\mbox{\small 63}}\! \\ \hline
f(n) & /54\! & \!55\! & \!56\! & \!57\! & \!58\! & \!{\bf 59}/60\! & \!{\bf 60}/61\! & \!{\bf 61}/62\! & \!62/63\! & \!64\! & /65\! & \!66\! & \!67\! & \!68\! & \!69\! & \!70\! & \!71\! & \!72\! \\ \hline
Ref. & \!{\small\cite{NoBa}}\! & \!{\small\cite{WuSR}}\! & \!{\small\cite{ZhaCC3}}\! & \multicolumn{2}{|c|}{\small\cite{Par3}} & \multicolumn{3}{|c|}{{\bf *}{\small/\cite{Par3}}} & \!{\small\cite{Chen}/\cite{Par3}}\! & \!{\small\cite{WuSZR}}\! & \!{\small\cite{Par3}}\! & \multicolumn{7}{|c|}{\small\cite{WuSZR}} \\ \hline
\end{array}\right.$$
$$\left.\begin{array}{|c|c|c|c|c|c|c|c|c|c|c|c|c|c|c|c|c|c|c|c|c|}\hline
n & \!{\mbox{\small 64}}\! & \!{\mbox{\small 65}}\! & \!{\mbox{\small 66}}\! & \!{\mbox{\small 67}}\! & \!{\mbox{\small 68}}\! & \!{\mbox{\small 69}}\! & \!{\mbox{\small 70}}\! & \!{\mbox{\small 71}}\! & \!{\mbox{\small 72}}\! & \!{\mbox{\small 73}}\! & \!{\mbox{\small 74}}\! & \!{\mbox{\small 75}}\! & \!{\mbox{\small 76}}\! & \!{\mbox{\small 77}}\! & \!{\mbox{\small 78}}\! & \!{\mbox{\small 79}}\! & \!{\mbox{\small 80}}\! & \!{\mbox{\small 81}}\! & \!{\mbox{\small 82}}\! \\ \hline
f(n) & \!73\! & \!74\! & \!75\! & \!{\bf 76}\! & \!77\! & {\bf 78}/ & \!79\! & \!{\bf 80}/81\! & \!81/82\! & \!83\! & \!84\! & \!85\! & \!86\! & \!87\! & \!88\! & \!89\! & \!90\! & \!91\! & \!92\! \\ \hline
Ref. & \!{\small\cite{Par3}}\! & \!{\small\cite{WuSZR}}\! & \!{\small\cite{ZhaCC2}}\! & {\bf *} & \!{\small\cite{ZhaCC2}}\! & \!{\bf *}\! & \!{\small\cite{ZhaCC2}}\! & \!{\bf *}{\small/\cite{Par3}}\! & \!{\small\cite{Chen}/\cite{Par3}}\! & \multicolumn{8}{|c}{\small\cite{WuSZR}}\! & \multicolumn{2}{|c|}{\small\cite{Par3}} \\ \hline
\end{array}\right.$$
\end{center}

Given the equivalence between $f(n)$ and $R(C_4, W_n)$, some references in the previous tables pertain to bounds or exact values of
 $R(C_4, W_n)$.

 For $n\in J=\{34,35,36,37,38,39,43\}$, let $H_n$ denote the House of Graphs graph with identifier $53036$, $56941$, $56942$, $56943$, $56944$, $56945$, and $52632$, respectively~\cite{HoG}. The following results were computationally verified:
\begin{lemma} \label{infe}
$H_n\nsupseteq C_4$, for $n\in J$. Additionally, $\Delta(\overline{H_{34}})=27$, $\Delta(\overline{H_{35}})=28$, $\Delta(\overline{H_{36}})=29$, $\Delta(\overline{H_{37}})=30$, $\Delta(\overline{H_{38}})=31$, $\Delta(\overline{H_{39}})=32$ and $\Delta(\overline{H_{43}})=36$.
\end{lemma}

We now present the main result of this section:

\begin{theorem}
$f(27)=33$, $f(28)=35$, $f(29)=36$, $f(30)=37$, $f(31)=38$, $f(32)=39$, $f(33)=40$, $f(37)=44$, and $f(67)=76$.
\end{theorem}

\begin{proof}
By Lemma \ref{infe}, we have $f(28)\ge 35$, $f(29)\ge 36$, $f(30)\ge 37$, $f(31)\ge 38$, $f(32)\ge 39$, $f(33)\ge 40$, and $f(37)\ge 44$. Additionally, by Lemma \ref{Chen}, $f(27)\ge f(28)-2\ge 33$.

Suppose there exists a graph $G$ with 33 vertices such that $G\nsupseteq C_4$ and $\overline{G}\nsupseteq K_{1,27}$. Then, for every $v \in V(G)$, we have $d_{\overline{G}}(v)\le 26$, implying $d_G(v)=33-1-d_{\overline{G}}(v)\ge 6$. Consequently, $2e(G)=\sum_{v\in V(G)}d_G(v)\ge 198$. This leads to a contradiction, as the maximum number of edges in a graph with 33 vertices that does not contain $C_4$ is 96 \cite{AfMcK}. Therefore, $f(27)=33$.

For $28\le n\le 33$, we have $\left\lceil\sqrt{n}\right\rceil=6$, and by Corollary \ref{cop3}, $f(n)\le n+7$. Thus, $f(28)=35$, $f(29)=36$, $f(30)=37$, $f(31)=38$, $f(32)=39$, and $f(33)=40$.

By Corollary \ref{cop3}, $f(37)\le 44$, so $f(37)=44$.

According to \cite{WuSZR}, $f(76)=86$, and by Corollary 7, $76 \le f(153-f(76)) = f(67)$. By Theorem 4, $f(67)\le 76$. Therefore $f(67)=76$.
\end{proof}

The remaining bounds for $f(n)$ for small values of $n$ are as follows:

\begin{proposition}
$f(51)\ge 59$, $f(52)\ge 60$, $f(53)\ge 61$, $f(69)\ge 78$, and $f(71)\ge 80$.
\end{proposition}

\begin{proof}
If $58\le n\le 61$, then by [15], $f(n)=n+9$. Hence, by Corollary 7, $
n\le f(2n+1-f(n))=f(n-8)$. Thus, $f(51)\ge 59$, $f(52)\ge 60$, and $f(53)\ge 61$.

If $77\le n\le 80$, then by [15], $f(n)=n+10$. Hence, by Corollary 7,
$n\le f(2n+1-f(n))=f(n-9)$. Thus, $f(69)\ge 78$ and $f(71)\ge 80$.
\end{proof}

\begin{remark}
{\em Based on the previous results, if $3\le n\le 39$, then $f(n)\ge f(n-1)+1$, and if $2\le n\le 82$, then $f(n)\ge n+\lceil\sqrt{n}\rceil$. No counterexamples to these inequalities are known for larger values of $n$.}
\end{remark}

\end{document}